\documentclass[12pt]{article}
 \usepackage{amsmath,amssymb,amsthm} %% loaded by `maa-monthly'

\usepackage{url} %% for a Wikipedia reference

\title{Legendre's Singular Modulus}

\author{Mark B. Villarino\\
Escuela de Matem\'atica, Universidad de Costa Rica,\\
11501 San Jos\'e, Costa Rica}

\date{}

\theoremstyle{plain}
\newtheorem{thm}{Theorem}[section]   %% Theorem 3.1
\newtheorem{cor}[thm]{Corollary}     %% Corollary 3.2
\newtheorem{lemma}[thm]{Lemma}       %% Lemma 3.3

\theoremstyle{definition}
\newtheorem{defn}{Definition}[section] %% Definition 3.1

\numberwithin{equation}{section}

\makeatletter
\def\section{\@startsection{section}{1}{\z@}{-3.5ex plus -1ex minus
			  -.2ex}{2.3ex plus .2ex}{\large\bf}}
\def\subsection{\@startsection{subsection}{2}{\z@}{-3.25ex plus -1ex
			  minus -.2ex}{1.5ex plus .2ex}{\normalsize\bf}}
\makeatother

\DeclareMathOperator{\cn}{cn}   %% Jacobi cn function
\DeclareMathOperator{\dn}{dn}   %% Jacobi dn function
\DeclareMathOperator{\sn}{sn}   %% Jacobi sn function

\renewcommand{\leq}{\leqslant}  %% (to save typing)

\newcommand{\xx}{\mathbf{x}}    %% vector x

\newcommand{\half}{{\mathchoice{\thalf}{\thalf}{\shalf}{\shalf}}}
\newcommand{\shalf}{{\scriptstyle\frac{1}{2}}} %% tiny fraction  1/2
\newcommand{\thalf}{\tfrac{1}{2}} %% small fraction  1/2

\newcommand{\word}[1]{\quad\mbox{#1}\quad} %% well-spaced words

\newcommand{\hideqed}{\renewcommand{\qed}{}} %% to suppress `\qed'

\begin{document}

\maketitle

\section{Introduction} % 1
\label{sec:intro}

Two characteristics of mathematics charm and delight most professional
mathematicians.

The first is its \emph{historical continuity}. For example, although
Euclid created his proofs of Pythagoras's theorem \cite[I.47]{Euclid}
and of the infinitude of primes \cite[IX.20]{Euclid} some 2300 years
ago, they remain as fresh, as compelling, and as beautiful today as
they were when he first wrote them down.

The second is the way that seemingly disparate areas of mathematics
reveal \emph{deep and unsuspected relationships}. For example the
number $\pi$ appears in a myriad distant parts of mathematics. One
might say that $\pi$ is \emph{ubiquitous}. A personal favorite of ours
is the fact, discovered by Dirichlet \cite[Thm.~332]{HardyW}, that
\emph{the probability that two integers taken at random are relatively
prime is $\dfrac{6}{\pi^2}$}. What on earth does $\pi$, the universal
ratio of the circumference to the diameter of any circle, have to do
with the common divisors of two integers taken at random? Dirichlet's
result shows that the relationship is profound. Indeed it is based on
Euler's solution to the Basel problem, i.e., that
$\sum_{n=1}^\infty \frac{1}{n^2} = \frac{\pi^2}{6}$, which, in turn,
is based on the fact that the non zero roots of the transcendental
equation $\sin x = 0$ are $\pm \pi$, $\pm 2\pi$, $\pm 3\pi$, \dots

Our present paper is devoted the unexpected and fascinating ubiquity
of Legendre's relatively unknown (third) \emph{singular modulus}
$k := \sin \frac{\pi}{12} = \frac{1}{2} \sqrt{2 - \sqrt{3}}$ (see the
definition below). These instances include:
\begin{itemize}
\item
Legendre's original proof of the \emph{first appearance} in the
history of mathematics of a singular modulus;
\item
Ramanujan's formula for the arc length of an ellipse of eccentricity
$\sin \frac{\pi}{12}$;
\item
the three-body choreography on a lemniscate.
\end{itemize}

We also briefly mention random walks on a cubic lattice and simple
pendulum renormalization.

These occurrences are familiar to most workers in these areas but not
to the general mathematical public at large. They deserve to be better
known since they display the two characteristics we listed above and
they exhibit quite beautiful mathematics. We hope that our paper makes
them easily available.

%%%%%%%%%%%%%%%%%%%%%%%%%%%%%%%%%%%%%%%%%%%%%%%%%%%

\section{Legendre's Singular Modulus}
\label{sec:Legendre}

In 1811, the famous french mathematician Adrien-Marie Legendre
published the following notable result \cite[p.~60]{Legendre}, which
we present in a form close to the spirit of Ramanujan.

\begin{thm} % 2.1
\label{th:Legendre}
If
$$
f(\alpha) := 1 + \biggl( \frac{1}{2} \biggr)^2 \alpha 
+ \biggl( \frac{1\cdot 3}{2\cdot 4} \biggr)^2 \alpha^2
+ \biggl( \frac{1\cdot 3\cdot 5}{2\cdot 4\cdot 6} \biggr)^2 \alpha^3
+\cdots
$$
and
$$
f(1 - \alpha) = \sqrt{3} \cdot f(\alpha),
$$
then
$$
\boxed{\alpha = \sin^2 \frac{\pi}{12} = \frac{1}{4}(2 - \sqrt{3}).}
$$
\end{thm}

We we note that the series converges for $-1\leq \alpha<1$ and that the quotient $\frac{f(1-\alpha)}{f(\alpha)}$ decreases monotonically from $\infty$ to $1$ as $\alpha$ increases from $0$ to $1$.  Thus there is a unique $\alpha$ such that $f(1 - \alpha) = \sqrt{3} \cdot f(\alpha)$ is true.

We now formulate Legendre's own statement \cite[p.~59]{Legendre}.  Let 
$0 \leq k \leq 1$ and
$$
K(k) := \int_0^{\pi/2} \frac{d\theta}{\sqrt{1 - k^2 \sin^2\theta}}
\word{and}  K' := K(\sqrt{1 - k^2}).
$$
Then, if we expand the integrand by the binomial theorem, integrate
term by term, and write $k^2$ for $\alpha$, we obtain Legendre's
original form.

\begin{thm} % 2.2
\label{th:original-form}
\begin{equation}
\label{eq:CM} % (2.1)
\boxed{\frac{K'(k)}{K(k)} = \sqrt{3} 
\implies k^2 = \sin^2 \frac{\pi}{12} = \frac{1}{4}(2 - \sqrt{3})}
\end{equation}
\end{thm}

The equation \eqref{eq:CM} is the first example of a \emph{singular
modulus} and we define the concept below.

\begin{defn} % 2.1
\label{df:singular-modulus}
If $N$ is a positive integer and the following equation holds:
\begin{equation}
\label{eq:CM1} % (2.2)
\frac{K'(k)}{K(k)} = \sqrt{N},
\end{equation}
then $k$ is called the \textbf{singular modulus} for~$N$.
\end{defn}

There is an enormous literature dealing with singular moduli. Today
the branch of mathematics which studies them is called \emph{complex
multiplication} \cite{Cox}. Legendre's relation \eqref{eq:CM} is the
first published explicit example of it.

The great German mathematician, David Hilbert, remarked that the theory of complex multiplication  "...which brings together number theory and analysis, is not only the most beautiful part of mathematics, but of all of science."~\cite{Reid}, (p. 200)  The theory was born when Abel~\cite{Abel}  published the following general theorem.
\begin{thm}
Whenever
$$\frac{K'}{K}=\frac{a+b\sqrt{N}}{c+d\sqrt{N}}$$
where $a,b,c,d$ are integers, $k$ is the root of an algebraic equation with integral coefficients which is solvable by radicals.
\end{thm}
The group of the equation turns out to be abelian, which is why it is solvable by radicals.  The splitting field is therefore called an \emph{abelian extension}.
Thus the algebraic nature of the modulus is deeply related to its arithmetic properties.  Indeed, Gauss' theory of cyclotomy, in which the cyclotomic equation has an abelian group, and thus is solvable by radicals, lead Abel by analogy to a corresponding theory of using a ruler and compass to divide a \emph{lemniscate} into equal arcs, and thus the above theorem.    Kronecker suggested that since the roots of unity, which are values of the exponential function, generate abelian extensions of the rational number field, perhaps also\emph{ elliptic functions could generate abelian extensions of quadratic number fields}, his so-called "Jugentraum."  Years of work showed that Kronecker was basically right and this theory of "singular moduli," a term introduced by Kronecker, would create wonderful mathematics.  
A simple example is the Heegner/Stark proof that \emph{there are only nine imaginary quadratic number fields which admit unique factorization}, a problem posed by Gauss some 170 years earlier.  Or the transcendental proof of the\emph{ biquadratic} reciprocity theorem.  Finally it was natural to consider abelian extensions of arbitrary algebraic number fields which lead to today's profound theories called "class field theory."
It is amazing that transcendental functions lead to the solution of deep problems in discrete number theory and commutative algebra, \emph{all brought about by Legendre's original singular modulus and Abel's brilliant generalization.}

All published proofs of \eqref{eq:CM}, except for Legendre's own,
either use Jacobi's theory of the cubic transformation and the modular
equation of degree three \cite[p.~188]{Cayley}, or the complex
variable proof given in \cite[pp.~525--526]{WhittakerW} or
\cite[p.~92]{Bowman}.

Legendre's original proof, on the other hand, is completely elementary
and we have not been able to find a presentation of it in the
literature. Perhaps the reason is that some fifteen years later Jacobi
and Abel made Legendre's almost forty years of work obsolete by
studying the elliptic \emph{functions} instead of integrals and much
of his work was subsequently neglected. Still, it is a brilliant
\emph{tour de force} in first-year integral calculus and deserves to
be better known. Unfortunately it is seemingly unmotivated. Indeed, we
suggest that Legendre discovered his result by accident and then
developed his proof, which is a verification. Even so, we will present
it here. It is based on his theory of the \emph{bisection and
trisection} of elliptic integrals.

To this end we will first briefly review those elementary properties
of (real) elliptic integrals which Legendre uses in his proof. Then
will give a detailed presentation of Legendre's proof itself.

%%%%%%%%%%%%%%%%%%%%%%%%%%%%%%%%%%%%%%%%%%%%%%%

\subsection{Elliptic Integrals}

\begin{defn} % 2.2
\label{df:elliptic-integral}
Let $k$ be a real number such that $0 \leq k \leq 1$ and $\Phi$ a real
number such that $0 \leq \Phi \leq \frac{\pi}{2}$. Then
$$
F(\Phi,k) := \int_0^\Phi \frac{d\phi}{\sqrt{1 - k^2 \sin^2\phi}}
$$
is called an \textbf{elliptic integral} (of the first kind), $\Phi$ is
the \textbf{amplitude} and $k$ is the \textbf{modulus}. Usually we
write $F(\Phi)$ for brevity if the modulus is not important.
\end{defn}

Legendre proposed the following \emph{trisection} problem:
\textit{If $k$ is given, it is required to find the amplitude $\Phi$
which solves the following equation}:
\begin{equation}
\label{eq:tri} % (2.3)
F(\Phi) \equiv \int_0^\Phi \frac{d\phi}{\sqrt{1 - k^2 \sin^2\phi}}
= \frac{1}{3} \int_0^{\pi/2} \frac{d\psi}{\sqrt{1 - k^2 \sin^2\psi}}
\equiv \frac{K}{3}.
\end{equation}

Legendre solves it by proving (after Euler) that $F(\Phi)$ satisfies
an \emph{addition theorem} \cite[p.~20]{Legendre}: namely
\begin{equation}
\label{eq:aat1} % (2.4)
F(\Phi) + F(\Psi) = F(\mu),
\end{equation}
where
$$
\cos \mu 
= \cos\Phi \cos\Psi - \sin\Phi \sin\Psi \sqrt{1 - k^2 \sin^2\mu}.
$$

Taking $\mu = \frac{\pi}{2}$ and $F(\Psi) := F(\Phi) + F(\Phi)$ in
\eqref{eq:aat1}, after some algebra and trigonometric reduction
Legendre \cite[p.~29]{Legendre} finds the following result.

\begin{thm} % 2.3
\label{th:Legendre-eqn}
Let $k$ be fixed. Then $x := \sin\Phi$ in \eqref{eq:tri} is a root of
\begin{equation}
\label{eq:tri1} % (2.5)
k^2 x^4 - 2k^2 x^3 + 2x - 1 = 0.
\end{equation}
\end{thm}

He then investigates two special cases.

\begin{cor} % 2.4
\label{cr:Legendre-eqn}
\begin{enumerate}
\item % (a)
If $k := \half \sqrt{2 + \sqrt{3}} = \cos \frac{\pi}{12}$, then
$$
\tan\Phi = \sqrt{\frac{2}{\sqrt{3}}}
$$
and vice versa.
\item % (b)
If $k := \half \sqrt{2 - \sqrt{3}} = \sin \frac{\pi}{12}$, then
$$
\cos\Phi = (2^{2/3} - 1) \sqrt{\frac{2 + \sqrt{3}}{\sqrt{3}}}
$$
and vice versa.
\end{enumerate}
\end{cor}

The proofs are elementary and only involve high-school algebra and
trigonometry.

Along the way, Legendre also solves the general \textit{bisection
problem} \cite[p.~25]{Legendre}. (We follow Beenakker
\cite{Beenakker}.)

\begin{thm} % 2.5
\label{th:bisection}
If
$$
\sin\Phi 
= \frac{\sin \frac{\theta}{2}}{\sqrt{\half + \half\Delta(\theta})},
\quad
\Delta(\theta) := \sqrt{1 - k^2 \sin^2\theta},
$$
then
$$
F(\Phi,k) = \half F(\theta,k).
$$
\end{thm}

It is convenient to prove the following technical lemma (we follow
Bowman \cite[p.~91]{Bowman}).

\begin{lemma} % 2.6
\label{lm:Bowman}
If $T := 1 + 2t^2 \cos 2\alpha + t^4$, then
\begin{equation}
\label{eq:tl} % (2.6)
\int_0^x \frac{dt}{\sqrt{T}} = \int_{1/x}^\infty \frac{dt}{\sqrt{T}}
= \frac{1}{2} \int_0^{2x/(1 + x^2)} 
\frac{dy}{\sqrt{(1 - y^2)(1 - k^2 y^2)}}
\end{equation}
where $0 < x < 1$ and $0 < \alpha < \frac{\pi}{2}$ and
$k := \sin\alpha$.
\end{lemma}

\begin{proof}
That the first two integrals are equal follows by making the substitution
$t' := 1/t$ in either of them.

Let $u$ denote the first integral. Putting $t := \tan\theta$, followed
by $y := \sin 2\theta$ we find after some algebra that
$$
u = \int_0^{\arctan x} 
\frac{d\theta}{\sqrt{1 - \sin^2\alpha \sin^2 2\theta}}
= \frac{1}{2} \int_0^{2x/(1 + x^2)} 
\frac{dy}{\sqrt{(1 - y^2)(1 - k^2 y^2)}}\,.
\eqno \qed
$$
\hideqed
\end{proof}

\begin{cor} % 2.7
\label{cr:Bowman}
If $k := \sin\alpha$, then
$$
\int_0^\infty \frac{dt}{\sqrt{T}} = K(k).
$$
\end{cor}

\begin{proof}
Take $x = 1$ in \eqref{eq:tl}.  Then
\begin{align*}
\int_0^\infty \frac{dt}{\sqrt{T}}
&= \int_0^1 \frac{dt}{\sqrt{T}} + \int_1^\infty \frac{dt}{\sqrt{T}}
\\
&= \half K + \half K =  K.
\tag*{\qed}
\end{align*}
\hideqed
\end{proof}

Legendre \cite[p.~31]{Legendre} notes the following special cases of
this corollary.

\begin{cor} % 2.8
\label{cr:Bowman-cases}
\begin{alignat*}{2}
& (\mathrm{a}) \quad k=\sin(15^o)\Rightarrow
& \int_0^\infty \frac{dt}{\sqrt{1 + \sqrt{3}t^2 + t^4}}
&= K\biggl( \frac{1}{2} \sqrt{2 - \sqrt{3}} \,\biggr),
\\
& (\mathrm{b}) \quad k=\sin(75^o)\Rightarrow
& \int_0^\infty \frac{dt}{\sqrt{1 - \sqrt{3}t^2 + t^4}}
&= K\biggl( \frac{1}{2} \sqrt{2 + \sqrt{3}} \,\biggr).
\end{alignat*}
\end{cor}

%%%%%%%%%%%%%%%%%%%%%%%%%%%%%%%%%%%%%%%%%%%

\subsection{Legendre's original proof}

We will closely follow Legendre \cite[pp.~59--60]{Legendre} except for
one step. He computes the integral
\begin{equation}
\label{eq:R} % (2.7)
R := \int_0^1 \frac{dx}{(1 - x^3)^{2/3}}
\end{equation}
in two ways. Then he equates the two expressions and 
obtains~\eqref{eq:CM}.

\begin{lemma} % 2.9
\label{lm:Legendre-R-first}
On the one hand,
$$
R = \frac{4}{3\cdot 4^{1/3} 3^{1/4}}
K\Biggl( \sqrt{\frac{2 + \sqrt{3}}{4}} \Biggr).
$$
% where $k^2 = \frac{2 + \sqrt{3}}{4}$.
\end{lemma}

\begin{proof}
First make the change of variable
$$
\biggl( \frac{x}{y} \biggr)^{3/2} := 1 - x^3,
$$
and $R$ becomes
$$
R = \int_0^\infty \frac{dy}{\sqrt{4y^3 + 1}}\,.
$$

Then make the change of variable 
$$
4^{1/3} y =: z^2 - 1
$$
which transforms $R$ into
\begin{equation}
\label{eq:R1} % (2.8)
R = \frac{2}{4^{1/3}} \int_1^\infty \frac{dz}{\sqrt{z^4 - 3z^2 + 3}}.
\end{equation}
At this point Legendre makes a complicated change of variable which
obscures what is going on. So we ask the reader's indulgence as we
depart slightly from his development. We write
\begin{equation}
\label{eq:R2} % (2.9)
R = \frac{2}{4^{1/3}} \int_0^\infty \frac{dz}{\sqrt{z^4 - 3z^2 + 3}}
- \frac{2}{4^{1/3}} \int_0^1 \frac{dz}{\sqrt{z^4 - 3z^2 + 3}}
\end{equation}
and we make the change of variable $3^{1/4} t := z$ to obtain
\begin{equation}
\label{eq:R4} % (2.10)
R = \frac{2}{4^{1/3} 3^{1/4}} \int_0^\infty 
\frac{dt}{\sqrt{t^4 - \sqrt{3}t^2 + 1}}
- \frac{2}{4^{1/3} 3^{1/4}} \int_0^{3^{-1/4}}
\frac{dt}{\sqrt{t^4 - \sqrt{3}t^2 + 1}}\,.
\end{equation}

Then by Corollary~\ref{cr:Bowman-cases}(b), 
\begin{equation}
\label{eq:R3} % (2.11)
R = \frac{2}{4^{1/3} 3^{1/4}} 
K\biggl( \frac{1}{2} \sqrt{2 + \sqrt{3}} \biggr)
- \frac{2}{4^{1/3} 3^{1/4}} \int_0^{3^{-1/4}}
\frac{dt}{\sqrt{t^4 - \sqrt{3}t^2 + 1}}\,.
\end{equation}
% where the modulus is $k^2 = \frac{2 + \sqrt{3}}{4}$.

We let $R_1$ denote the integral on the right hand side of \eqref{eq:R3},:
$$
R_1 := \int_0^{3^{-1/4}} \frac{dt}{\sqrt{t^4 - \sqrt{3}t^2 + 1}}
$$

Set $t =: \tan\theta$. Then
\begin{equation}
\label{eq:R5} % (2.12)
R_1 = \int_0^{\arctan 3^{-1/4}} 
\frac{d\theta}{\sqrt{1 - \frac{2 + \sqrt{3}}{4} \sin^2 2\theta}}
\end{equation}

We now apply Theorem~\ref{th:bisection} with
$k^2 = \frac{2 + \sqrt{3}}{4}$ and $\theta := 2\arctan 3^{-1/4}$ and
find that 
\begin{equation}
\label{eq:sinphi} % (2.13)
\sin\Phi = \frac{2\sin \frac{\theta}{2}}
{\sqrt{\sqrt{4 - (\sqrt{3} + 2) \sin^2\theta} + 2}} = \sqrt{3} - 1,
\end{equation}
and therefore
$$
\Phi = \arcsin (\sqrt{3} - 1) 
= \arctan \frac{\sqrt{2}}{\sqrt[4]{3}}\,;
$$
and we conclude that
\begin{equation}
\label{eq:id} % (2.14)
\int_0^{\arctan(1/\sqrt[4]{3})}
\frac{d\theta}{\sqrt{1 - \frac{2 + \sqrt{3}}{4}\sin^2 2\theta}}
= \int_0^{\arctan(\sqrt{2}/\sqrt[4]{3})}
\frac{d\phi}{\sqrt{1 - \frac{2 + \sqrt{3}}{4} \sin^2\phi}}\,.
\end{equation}

But Corollary~\ref{cr:Legendre-eqn} says that 
$R_1 = F(\Phi) = \frac{K}{3}$. Therefore, by~\eqref{eq:R3}:
\begin{align*}
R &= \frac{2}{4^{1/3} 3^{1/4}}
K\biggl( \frac{1}{2} \sqrt{2 + \sqrt{3}} \,\biggr)
- \frac{2}{4^{1/3} 3^{1/4}} \frac{K(\half \sqrt{2 + \sqrt{3}})}{3}
\\
&= \frac{4}{3\cdot 4^{1/3} 3^{1/4}}
K\biggl( \frac{1}{2} \sqrt{2 + \sqrt{3}} \,\biggr).
\qedhere
\end{align*}
\end{proof}

We encourage the reader to find a direct change of variable proof of
the identity~\eqref{eq:id}, something that seems quite difficult.

\begin{lemma} % (2.10)
\label{lm:Legendre-R-second}
On the other hand,
$$
R = \frac{4\sqrt{3}}{3\cdot 4^{1/3} 3^{1/4}}\,
K\Biggl( \sqrt{\frac{2 - \sqrt{3}}{4}} \,\Biggr).
$$
%where $k^2 = \frac{2-\sqrt{3}}{4}$.
\end{lemma}

\begin{proof}
We make the change of variable in the original
integral~\eqref{eq:R}:
$$
\biggl( 1 - \frac{x}{y} \biggr)^3 := 1 - x^3
$$
which gives
$$
R = \sqrt{3} \int_1^\infty \frac{dy}{\sqrt{4y^3 - 1}}\,.
$$

Then follows the change of variable
$$
4^{1/3} y =: z^2 + 1
$$
which gives
\begin{align*}
R &= \frac{2\sqrt{3}}{4^{1/3}} \int_{\sqrt{4^{1/3}-1}}^\infty
\frac{dz}{\sqrt{z^4 + 3z^2 + 3}}
\\
&= \frac{2\sqrt{3}}{4^{1/3}} \int_0^\infty 
\frac{dz}{\sqrt{z^4 + 3z^2 + 3}} 
- \frac{2\sqrt{3}}{4^{1/3}} \int_0^{\sqrt{4^{1/3}-1}}
\frac{dz}{\sqrt{z^4 + 3z^2 + 3}}
\\
&= \frac{2\sqrt{3}}{4^{1/3}} 
K\Biggl( \sqrt{\frac{2 - \sqrt{3}}{4}} \,\Biggr)
- \frac{2\sqrt{3}}{4^{1/3}} \int_0^{\sqrt{4^{1/3}-1}}
\frac{dz}{\sqrt{z^4 + 3z^2 + 3}}\,.
\end{align*}
From this point on the argument is identical with that of
Lemma~\ref{lm:Legendre-R-first}, only we use
Corollary~\ref{cr:Bowman-cases}(a) and
Corollary~\ref{cr:Legendre-eqn}(b). We leave the details to the
reader.
\end{proof}

\begin{proof}[Proof of Theorem~\ref{th:original-form}]
We now show Legendre's singular modulus relation~\eqref{eq:CM}.

We equate the two formulas for $R$ in Lemmas \ref{lm:Legendre-R-first}
and~\ref{lm:Legendre-R-second}, and observe the equality of the
complete elliptic integrals
$$
K\Biggl( \sqrt{\frac{2 + \sqrt{3}}{4}} \,\Biggr)
= K'\Biggl( \sqrt{\frac{2-\sqrt{3}}{4}} \,\Biggr),
$$
since
$$
\Biggl( \sqrt{\frac{2 - \sqrt{3}}{4}} \,\Biggr)^2
= 1 - \Biggl( \sqrt{\frac{2 + \sqrt{3}}{4}} \,\Biggr)^2.
$$
This completes Legendre's proof of Legendre's singular modular
relation~\eqref{eq:CM}.
\end{proof}

%%%%%%%%%%%%%%%%%%%%%%%%%%%%%%%%%%%%%%%%%%%%%%%%%

\section{Ramanujan's Ellipse} % 3
\label{sec:ellipse-perimeter}

At the end of one of his most famous papers \cite{Ramanujan},
Ramanujan states the next formula ---without proof~(!)

\begin{thm} % 3.1
\label{th:ellipse-perimeter}
If $a$ is the semimajor axis of an ellipse whose eccentricity is
$\sin \frac{\pi}{12} = \sin 15^\circ$, then the perimeter, $p$, of the
ellipse is
$$
p = a\,\sqrt{\frac{\pi}{\sqrt{3}}} \Biggl\{ 
\biggl( 1 + \frac{1}{\sqrt{3}} \biggr)
\frac{\Gamma(\frac{1}{3})}{\Gamma(\frac{5}{6})}
+ \frac{2\Gamma(\frac{5}{6})}{\Gamma(\frac{1}{3})} \Biggr\}.
$$
\end{thm}

We recall that the \emph{gamma function}, $\Gamma(x)$ is defined for all $x>0$ by the improper integral
$$\Gamma(x):=\int_{0}^{\infty}t^{x-1}e^{-t}~dt.$$

This beautiful formula of Ramanujan is remarkable for 
% various 
several
reasons.  

In the first place, in 1833 Liouville \cite{Liouville} had proved that
the formula for the arc length of an ellipse with eccentricity $k$,
namely,
$$
p = 4aE := 4a \int_0^{\pi/2} \sqrt{1 - k^2 \sin^2\theta} \,d\theta
$$
cannot be expressed as a finite combination of ``elementary''
functions and thus defines a new class of transcendentals. $E$~is
called a \textit{complete elliptic integral of the second kind}.
Yet, Ramanujan expresses it as a finite combination of gamma functions
of rational arguments.

Secondly, in 1949, Selberg and Chowla ``explained'' Ramanujan's result
when they proved \cite{ChowlaS} that $K$ could be expressed as a
\emph{finite product} of gamma functions of rational arguents if
$\frac{K'}{K} = \sqrt{N}$ for 
% any 
some 
integer~$N$. This had already been known since the time of Legendre
for $N = 1,2,3$. (The case $N = 3$ is the one discovered by Legendre
and is the ``third'' in order. So it is called the ``third'' singular
modulus.) On the face of it, this says nothing about a formula
for~$E$.

Finally, in 1811, 
%%% (which, by the way, comes before 1949 :-)
by another \emph{tour de force}, Legendre \cite[p.~60]{Legendre}
proved the relation that allowed the calculation of $E$ from that
of~$K$. Namely, if $k = \sin \frac{\pi}{12}$, then
\begin{equation}
\label{eq:KE} % (3.1)
\frac{\pi}{4\sqrt{3}}
= K \cdot \biggl\{ E - \frac{\sqrt{3} + 1}{2\sqrt{3}}\,K \biggr\}.
\end{equation}

In order to prove Ramanujan's formula, we first express 
$K(\sin 15^\circ)$ in terms of the Gamma function.

\begin{thm} % 3.2
\label{th:K-gamma}
\begin{equation}
\label{eq:K-gamma} % (3.2)
K(\sin 15^\circ) = \frac{1}{2\sqrt{3}} \sqrt{\frac{\pi}{\sqrt{3}}}
\cdot \frac{\Gamma(\frac{1}{6})}{\Gamma(\frac{2}{3})}\,.
\end{equation}
\end{thm}

\begin{proof}
Our proof freely uses the properties of the Gamma and Beta
functions~\cite{WhittakerW}:
\begin{equation}
\label{eq:beta} % (3.3)
B(x,y) := \frac{\Gamma(x)\Gamma(y)}{\Gamma(x + y)}
\end{equation}
where $x > 0$, $y > 0$; and expresses Legendre's original
integral~\eqref{eq:R} by~\eqref{eq:beta}. It is known that
\begin{equation}
\label{eq:beta1} % (3.4)
B(x,y) = \int_0^1 t^{x-1} (1 - t)^{y-1} \,dt
= n \int_0^1 s^{nx-1} (1 - s^n)^{y-1} \,ds,
\end{equation}
for $n > 0$, under the change of variable $t =: s^n$.
%%% That's not so well known, actually.
%%% I found nothing like it in Andrews, Askey and Roy. -- Joe

If we take $n = 3$, $x = y = \frac{1}{3}$ in \eqref{eq:beta1} and
apply \eqref{eq:beta}, we obtain the following formula for Legendre's
original integral in terms of gamma functions:
\begin{equation}
\label{eq:Rgamma} % (3.5)
R = \int_0^1 \frac{ds}{(1 - s^3)^{2/3}}
= \frac{1}{3} B\biggl( \frac{1}{3}, \frac{1}{3} \biggr)
= \frac{1}{3}\,
\frac{\Gamma(\frac{1}{3})\Gamma(\frac{1}{3})}{\Gamma(\frac{2}{3})}\,.
\end{equation}
Now we apply Legendre's \emph{duplication formula}~\cite{WhittakerW}, (p.240), with
$x = \frac{1}{6}$ to obtain
\begin{equation}
\label{eq:dup} % (3.6)
\Gamma\biggl( \frac{1}{6} \biggr) \Gamma\biggl( \frac{2}{3} \biggr)
= 2^{2/3} \sqrt{\pi} \cdot \Gamma\biggl( \frac{1}{3} \biggr),
\end{equation}
and then the \emph{reflection formula}~\cite{WhittakerW}, (p.239), to obtain
$$
\Gamma\biggl( \frac{2}{3} \biggr) 
= \frac{\pi}{\frac{\sqrt{3}}{2}\,\Gamma(\frac{1}{3})}\,.
$$

Substituting this into \eqref{eq:dup}, we obtain:
\begin{equation}
\label{eq:Rgamma1} % (3.7)
\Gamma^2\biggl( \frac{1}{3} \biggr)
= \frac{2^{1/3}}{\sqrt{3}} \cdot \sqrt{\pi} 
\cdot \Gamma\biggl( \frac{1}{6} \biggr),
\end{equation}
and substituting \eqref{eq:Rgamma1} into~\eqref{eq:Rgamma} gives us
the final form of Legendre's original integral in terms of gamma
functions:
\begin{equation}
\label{eq:Rgamma2} % (3.8)
R = \frac{1}{3} \sqrt{\frac{\pi}{3}} \cdot 2^{1/3}
\cdot \frac{\Gamma(\frac{1}{6})}{\Gamma(\frac{2}{3})}
\end{equation}
and finally, substituting \eqref{eq:Rgamma2} into
Lemma~\ref{lm:Legendre-R-second} and solving for
$K(15^\circ)$ gives us Legendre's formula~\eqref{eq:K-gamma}.
\end{proof}

We now give a simple proof of Ramanujan's formula for the perimeter of
an ellipse, since we have never seen one published.

\begin{proof}[Proof of Theorem~\ref{th:ellipse-perimeter}]
We use the reflection formula for the Gamma function,
$$
\Gamma(x) \Gamma(1 - x) = \frac{\pi}{\sin \pi x}
$$
with $x = \frac{1}{3}$ and again with $x = \frac{1}{6}$
in~\eqref{eq:K-gamma}; and we obtain
$$
K = \frac{1}{2} \sqrt{\frac{\pi}{\sqrt{3}}}\,
\frac{\Gamma(\frac{1}{3})}{\Gamma(\frac{5}{6})}\,.
$$
Then, substituting this into the relation~\eqref{eq:KE} of Legendre
and using $p = 4aE$, after some algebra we obtain the formula of
Ramanujan.
\end{proof}

%%%%%%%%%%%%%%%%%%%%%%%%%%%%%%%%%%%%%%%%%%%%%%%%

\section{Three-body choreography on a lemniscate} % 4
\label{sec:choreography}

The celebrated (unsolvable) \emph{three-body problem} (see \cite{wiki}
and the references listed therein) has challenged mathematicians for
over 350 years. We formulate it as follows.

\begin{quote}
\normalsize\itshape
Three point masses under the action of Newtonian gravity have
prescribed initial positions and velocities. It is required to
determine their positions and velocities at all later times.
\end{quote}

Mathematically, the problem reduces to the solution of nine nonlinear
coupled second-order ordinary differential equations. In general, no
closed-form solution exists since the resulting dynamical system is
\emph{chaotic} for most initial conditions.

However there are families of \emph{periodic} solutions for certain
special cases.

\begin{enumerate}
\item % (a)
In 1767, Leonhard Euler found three families of periodic solutions in
which the three masses are \emph{collinear}.

\item % (b)
In 1772, Lagrange found a family of solutions in which the three
masses form an \emph{equilateral triangle} at each instant.

\item % (c)
Over two hundred years passed (!) until Cris Moore~\cite{Moore},
in~1993, caused something of a sensation when he numerically
discovered a zero-momentum solution with three equal masses moving
around a \emph{figure-eight} shape orbit. In 2000 Alain Chenciner and
Richard Montgomery \emph{mathematically proved its formal existence}
\cite{ChencinerM}. Again, we emphasize that it was the first real
concrete example of a three-body orbital periodical solution in two
centuries!
\end{enumerate}

Since 2000, many researchers have numerically found other $n$-body
periodic solutions and today they number in the thousands! But they
all share the unhappy property of being \emph{unstable}, whereas the
figure-eight orbit has been numerically shown to \emph{be} stable (at
least for small perturbations).   Fujiwara, Fukuda, and Ozaki  ~\cite{FujiwaraFO1} beautifully related the figure-eight solution to the dihedral group, $D_6$, of regular hexagons and showed that the bifurcations of every figure-eight solutions are determined by the irreducible representations of $D_6$. The Catalan scientist Carles Sim\'o
\cite{Simo} coined the term \emph{choreography} to mean a periodic
motion on a closed orbit, where $N$ bodies chase each other in this
orbit with \emph{equal time-spacing}.

Every student of mathematics encounters a famous curve that looks like
a figure-eight, namely Bernoulli's \emph{lemniscate}.

Although the figure-eight orbit of Moore, Chenciner and Montgomery is
\emph{not} a lemniscate, a lemniscate actually approximates the
figure-eight orbit with an error of about one part in one thousand,
and when the two are placed on the same computer screen it is
difficult to distinguish them. Thus it was quite natural to
investigate three-body choreography on a lemniscate, itself, and
Toshiaki Fujiwara, Hiroshi Fukuda and Hiroshi Ozaki published a
detailed report in their beautiful paper~\cite{FujiwaraFO}.

In order to describe their results and to make our paper as
self-contained as possible, we must review some simple elementary
properties of the real-valued Jacobian elliptic functions
\cite{Bowman}. Then we will show how Legendre's singular modulus
amazingly arises from the condition that \emph{the total momentum is
zero}.

%%%%%%%%%%%%%%%%%%%%%%%%%%%%%%%%%%%%%%%%%%%%%

\subsection{Elliptic Functions} 

Let $k$ be a fixed real number such that $0 \leq k \leq 1$ and let
\begin{align*}
u &:= \int_0^x \frac{dt}{\sqrt{(1 - t^2)(1 - k^2 t^2)}}
\\
K &:= \int_0^1 \frac{dt}{\sqrt{(1 - t^2)(1 - k^2 t^2)}} 
\end{align*}
where $-1 \leq x \leq 1$ and the square roots are positive.  Please note that this definition of $K$ is equivalent to that given in Legendre's original statement of his singular modulus.

Then $u$ is an odd function of $x$ which increases monotonically from
$-K$ to $K$ as $x$ increases from $-1$ to $1$. This implies that $x$
\emph{is an odd function of $u$ which increases momotonically from
$-1$ to $1$ as $u$ increases from $-K$ to $K$.} For historical reasons
$x$ is denoted $\sn u$ which means:
\begin{equation}
\label{eq:sn} % (4.1)
u =  \int_0^{\sn u}\frac{dt}{\sqrt{(1-t^2)(1-k^2t^2)}}
\end{equation}

Taking the derivative with respect to $x$ we obtain
\begin{align*}
\frac{du}{dx} &= \frac{1}{\sqrt{(1 - x^2)(1 - k^2 x^2)}}\,,
\\
\frac{dx}{du} &= \sqrt{(1 - \sn^2 u)(1 - k^2 \sn^2 u)}\,.
\end{align*}

Cayley \cite[p.~8]{Cayley} calls $\sn u$ ``a sort of sine function''
and by analogy we define:
\begin{align*}
\cn u &:= \sqrt{1 - \sn^2 u},
\\
\dn u &:= \sqrt{1 - k^2 \sn^2 u},
\end{align*}
(which Cayley \cite[p.~8]{Cayley} calls ``sorts of cosine functions'')
where we take the positive square root so long as $u$ is confined to the interval $-K \leq u \leq K$ so
that $\cn$ and $\dn$ are \emph{even} functions of $u$. 
%%% Need to add the info that $\cn u$ can take negative values
We also note the following identity:
\begin{equation}
\label{eq:sc} % (4.2)
\sn^2 u + \cn^2 u = 1,
\end{equation}
and the following special values:
\begin{alignat*}{2}
\sn 0 &= 0, \quad  & \sn K &= 1,
\\
\cn 0 &= 1, \quad  & \cn K &= 0,
\\
\dn 0 &= 1, \quad  & \dn K &= \sqrt{1 - k^2} =: k'.
\end{alignat*}

Now we prove the fundamental identity.

\begin{thm}[Addition Theorem] % 4.1
\label{th:aat}
The following equation holds identically:
\begin{equation}
\label{eq:aat} % (4.3)
\sn(u + v) = \frac{\sn v \cn u \dn u + \sn u \cn v \dn v}
{1 - k^2 \sn^2u \sn^2v}\,.
\end{equation}
\end{thm}

\begin{proof}
Let $z(u,v)$ be the right hand side of~\eqref{eq:aat}. Then, a brute
force computation shows that
$$
\frac{\partial z}{\partial u} = \frac{\partial z}{\partial v}\,.
$$

This means that $z(u,v) = f(u + v)$ for some function $f$. Take
$v = 0$ in $z(u,v)$. We obtain
$$
f(u) \equiv \sn u.
\eqno \qed
$$
\hideqed
\end{proof}

The addition theorem allows us to prove the fundamental period
properties.

\begin{thm} % 4.2
\label{th:period-properties}
The following period equations hold:
\begin{align*}
\sn(u + 4K) &= \sn u 
\\
\cn(u + 4K) &= \cn u 
\\ 
\dn(u + 2K) &= \dn u 
\end{align*}
\end{thm}

\begin{proof}
$$
\sn(u + K) 
= \frac{\sn K \cn u \dn u + \sn u \cn K \dn K}{1 - k^2 \sn^2u \sn^2K}
= \frac{\cn u}{\dn u}
$$
where we used the definitions of $\cn$ and $\dn$, the identity
\eqref{eq:sc}, and the special values. Again we conclude that
%%% we need $\cn 2K = -1$, $\dn 2K = +1$ here
$$
\sn(u + 2K) = -\sn u
$$
and finally
$$
\sn(u + 4K) = \sn u.
$$
The other period equations are similar.
\end{proof}

This completes our introduction to the Jacobian elliptic functions.

%%%%%%%%%%%%%%%%%%%%%%%%%%%%%%%%%%%%%%%%%%%%%%%%%%

\subsection{Choreography on the Lemniscate}

Fujiwara, Fukuda and Ozaki \cite{FujiwaraFO} use the Jacobian elliptic
functions $\sn$ and $\cn$ to parametrize the lemniscate:
$$
(x^2 + y^2)^2 = x^2 - y^2
$$
as follows. The curve is given by the vector equation
\begin{equation}
\label{eq:x} % (4.4)
\xx(t) := x(t)\,\hat x + y(t)\,\hat y
\end{equation}
where $\hat x := (1,0)$ and $\hat y := (0,1)$ are the two orthogonal
base unit vectors defining the plane of the motion, and the coordinate
functions are explicitly given by:
$$
x(t) := \frac{\sn t}{1 + \cn^2 t}\,, \quad 
y(t) := \frac{\sn t \cn t}{1 + \cn^2 t}\,.
$$

This is a smooth periodic motion on the lemniscate with period
$$
T := 4K
$$
where $K$ is the complete elliptic integral of the first kind. The
positions of the choreographic three bodies are given by the vector:
$$
\{\xx_1(t),\xx_2(t),\xx_3(t)\}
= \biggl\{ \xx(t), \xx\biggl( t + \frac{4K}{3} \biggr),
\xx\biggl( t - \frac{4K}{3} \biggr) \biggr\}.
$$

What does all of this have to do with Legendre's singular modulus? Up
to this point everything we have said about the lemniscate works for
\emph{any} modulus~$k$. But \emph{now} Legendre enters. We posit that
\emph{the motion} \emph{must conserve the center of mass}.
Analytically this means that the following vector equation must hold
for all time~$t$:
\begin{equation}
\label{eq:ccm} % (4.5)
\xx(t) + \xx\biggl( t + \frac{4K}{3} \biggr) 
+ \xx\biggl( t - \frac{4K}{3} \biggr) = \mathbf{0}.
\end{equation}

The authors prove the following remarkable theorem.

\begin{thm} % 4.3
\label{th:choreography}
The equation \eqref{eq:ccm} holds if and only if the modulus is given
by
$$
\boxed{k^2 = \frac{2 + \sqrt{3}}{4} = \cos^2\frac{\pi}{12}.}
$$
\end{thm}

This means that the quarter-period, $K$, in the period of motion is
the $K'$ in Legendre's relation! So, some 209 years after Legendre
published his singular modulus, it reappears \emph{as a fundamental
property in a newly discovered choreography for the famous three-body
problem!}

We will give their proof (filling in many details) which uses the
properties of the Jacobian elliptic functions reviewed in the previous
section.

\begin{proof}[Proof of Theorem~\ref{th:choreography}]
We assume that \eqref{eq:ccm} holds.

First we establish a lemma which will be used in the final step in the
proof.

\begin{lemma} % 4.4
\label{lm:sn-value}
$\sn \frac{K}{3} = \sqrt{3} - 1$.
\end{lemma}

\begin{proof}
Since $\xx(K) = \hat x = (1,0)$, the equation \eqref{eq:ccm} shows us
that the $\hat x$-components of the other two masses must be 
$x\bigl( K \pm \frac{4K}{3} \bigr) = -\half$. Now
\begin{align*}
\xx(K) &= \hat x = (1,0)
\\
\implies x(K) &= 1
\\
\implies x\biggl( K - \frac{4K}{3} \biggr) 
&=  x\biggl( -\frac{K}{3} \biggr) = - \frac{1}{2}
\\
\implies x\biggl( \frac{K}{3} \biggr) &= \frac{1}{2}
\\
\implies \frac{\sn \frac{K}{3}}{1 + \cn^2 \frac{K}{3}} &= \frac{1}{2}
\\
\implies \sn \frac{K}{3}
&= \frac{1}{2} \biggl( 1 + \cn^2 \frac{K}{3} \biggr)
= 1 - \frac{1}{2} \sn^2 \frac{K}{3}
\\
\implies \sn \frac{K}{3} &= \sqrt{3} - 1
\end{align*}
where the last equation follows since $\sn \frac{K}{3} > 0$.
\end{proof}

Perhaps the reader would like to explore the relation between
Eq.~\eqref{eq:sinphi} and Lemma~\ref{lm:sn-value}.

We now evaluate $\sn \frac{10K}{3}$ in two ways. This will give us an
algebraic equation for the modulus~$k$.

\begin{lemma} % 4.5
\label{lm:sn-relation-one}
On the one hand,
$$
\sn \frac{10K}{3} = -\frac{\cn \frac{K}{3}}{\dn \frac{K}{3}}\,.
$$
\end{lemma}

\begin{proof}
\begin{align*}
\sn \frac{10K}{3} &= \sn\biggl( 3K + \frac{K}{3} \biggr)
\\
&= \frac{\sn \frac{K}{3} \cn 3K \dn 3K
+ \sn 3K \cn \frac{K}{3} \dn \frac{K}{3}}
{1 - k^2 \sn^2 3K \sn^2 \frac{K}{3}}
= - \frac{\cn \frac{K}{3}}{\dn \frac{K}{3}}
\end{align*}
since $\sn 3K = - \sn K = -1$; $\cn 3K = - \cn K = 0$, as can
easily be verified.
\end{proof}

\begin{lemma} % 4.6
\label{lm:sn-relation-two}
On the other hand,
$$
\sn \frac{10K}{3}
= -2\,\frac{\sn \frac{K}{3} \dn \frac{K}{3} \cn \frac{K}{3}}
{1 - k^2 \sn^4 \frac{K}{3}}\,.
$$
\end{lemma}

\begin{proof}
\begin{align*}
\sn \frac{10K}{3} &= \sn\biggl( 4K - \frac{2K}{3} \biggr)
\\
&= \frac{\sn\bigl(-\frac{2K}{3}\bigr) \cn 4K \dn 4K
+ \sn 4K \cn\bigl(-\frac{2K}{3}\bigr) \dn\bigl(-\frac{2K}{3}\bigr)}
{1 - k^2 \sn^2 4K \sn^2\bigl( -\frac{2K}{3} \bigr)}
\\
&= - \sn\biggl( -\frac{2K}{3} \biggr)
\\
&= -2\,\frac{\sn \frac{K}{3} \dn \frac{K}{3} \cn \frac{K}{3}}
{1 - k^2 \sn^4 \frac{K}{3}}
\end{align*}
since 
$$
\sn 4K = 0, \quad 
\sn 2u = 2\,\frac{\sn u \cn u \dn u}{1 - k^2 \sn^4 u}\,,
$$
as can easily be verified.
\end{proof}

Equating the two expressions and solving for $k^2$ we obtain
\begin{align*}
k^2 
&= \frac{1 - 2 \sn \frac{K}{3} \dn^2 \frac{K}{3}}{\sn^4 \frac{K}{3}}
\\
%%% maybe a step missing here
&= \frac{2 - \sqrt{3}}{(\sqrt{3} - 1)^4} = \frac{2 + \sqrt{3}}{4}\,.
\end{align*}

Conversely, taking $k^2$ given in the theorem gives us
$\sn \frac{K}{3} = \sqrt{3} - 1$.

This completes the proof of Theorem~\ref{th:choreography}.
% of Fujiwara, Fukuda and Ozaki.
\end{proof}

%%%%%%%%%%%%%%%%%%%%%%%%%%%%%%%%%%%%%%%%%%

\section{More Ubiquity} % 5
\label{sec:ubiquity}

Space prevents us from developing the relation between Legendre's
singular modulus and \emph{the probability that a random walk on a
cubic lattice returns to its origin} (but see \cite{Zucker} and the
references given there), a result that started in 1921 (when Polya
proved that the probability is \emph{not}~$1$ but without offering a numerical value) and was completed by
Joyce in 2003. We will simply state the result.

Let
\begin{equation*}
W := \frac{1}{\pi^3} \int_0^\pi \int_0^\pi \int_0^\pi
\frac{dx\,dy\,dz}{3 - \cos x\cos y - \cos y\cos z - \cos z\cos x}
\\
\end{equation*}
Watson~\cite{Watson} expressed $W$ in terms of Legendre's singular modulus:
\begin{equation*}
W= \frac{\sqrt{3}}{\pi^2}\,
K^2\Biggl( \sqrt{\frac{2 - \sqrt{3}}{4}} \,\Biggr).
\end{equation*}
Let
$$
W^+ :=  \frac{\sqrt{2}}{\pi^3} \int_0^\pi \int_0^\pi \int_0^\pi
\frac{dx\,dy\,dz}{3 + \cos x\cos y + \cos y\cos z + \cos z\cos x}\,.
$$

\begin{thm} % 5.1
\label{th:maybe-return}
The probability $p$ that a random walk returns to the origin on a
cubic lattice $\mathbb{Z}^3$ is
$$
p = 1 - \frac{1}{3W^+} = 0.3405373296\dots
$$
\end{thm}

So the integral which solves the random walk problem for a three
dimensional cubic lattice \emph{is obtained from the integral in
Legendre's relation by changing all the minus signs to plus and
multiplying by $\sqrt{2}$}~\cite{JoyceDZ}.
  
Nor do we have space to prove that the period of a simple pendulum
with amplitude $300^\circ$ is the same as that of a pendulum
\emph{three times as long} and which swings through an amplitude of
$60^\circ$ \cite[p.~98]{Bowman}. Indeed, every singular modulus has a
pendulum interpretation in which a given pendulum is replaced with
another \emph{with the same period but longer length and smaller
amplitude} which allows one to approximate the period with small
amplitude approximations. This is an example of the modern theory of
\emph{renormalization}.

At all events, we see that Legendre's singular modulus is (forgive the
pun) \emph{singularly} ubiquitous, while the wonderful historical
continuity of modern mathematics shows itself once again.

%%%%%%%%%%%%%%%%%%%%%%%%%%%%%%%%%%%%%%%%%%%%%%%%

\section*{Acknowledgements}

We thank Carlo Beenakker for his help in the proof of \eqref{eq:id},
and Joseph C. V\'arilly, and Adrian Barquero for useful comments. Financial support from
the Vicerector\'ia de Investigaci\'on of the University of Costa Rica
is acknowledged.

%%%%%%%%%%%%%%%%%%%%%%%%%%%%%%%%%%%%%%%%%%%%%%%

\end{document}